\documentclass [11pt] {amsart}
\usepackage{bbm}
\usepackage{amssymb}
\usepackage{epic}
\usepackage{ecltree}
\usepackage{tikz}
\newtheorem{theorem}{Theorem}
\newtheorem{lemma}[theorem]{Lemma}

\newtheorem{corollary}[theorem]{Corollary}

\newtheorem{question}[theorem]{Question}

\theoremstyle{definition}
\newtheorem{definition}[theorem]{Definition}

\newtheorem{remark}[theorem]{Remark}

\begin{document}
\title[$2$-rotund norms for Baernstein spaces and their duals]{$2$-rotund norms  for generalized\\  Baernstein spaces and their duals}



\author[S. J. Dilworth]{S. J. Dilworth}

\address{Department of Mathematics, University of South Carolina, Columbia,
SC 29208, USA}

\email{dilworth@math.sc.edu}


\author{Denka Kutzarova}

\address{Department of Mathematics, University of Illinois Urbana-Champaign,
Urbana, IL 61807, USA; Institute of Mathematics and Informatics, Bulgarian Academy of Sciences, Sofia, Bulgaria}

\email{denka@illinois.edu}


\thanks{ The first author was supported by Simons Foundation Collaboration Grant No. 849142.
 The second author was supported by Simons Foundation Collaboration Grant No. 636954.}


\keywords{$2$-rotundity, renorming, $2R$ norm,  Baernstein space, reflexivity.}

\subjclass[2020]{Primary 46B20; Secondary 46B03, 46B26.}

\begin{abstract} We consider a  generalized Baernstein space associated to  a compact family of finite subsets of an uncountable set.  We  show that for certain transfinitely defined families such spaces admit an 
equivalent $2$-rotund  norm. We also show that for an arbitrary  family the dual space admits an equivalent $2$-rotund  norm.
\end{abstract}

\maketitle  
\section{Introduction} \label{sec: Intro} The  notions of $2$-rotund and weakly $2$-rotund norms   were introduced by Milman \cite{M} and are defined as follows.
 \begin{definition}   Let $X$ be a Banach space. We say that a norm $\|\cdot\|$ on $X$ is $2$-rotund  ($2R$) (resp.\, weakly $2$-rotund ($W2R$))   if for every $(x_n) \subset X$ such that $\|x_n\| \le 1$ ($n \ge 1$) and
$$ \lim_{m,n \rightarrow \infty} \|x_m + x_n\| = 2,$$
there exists $x \in X$ such that $x = \lim_{n \rightarrow \infty} x_n$ strongly (resp.\, weakly). \end{definition}

It follows from a characterization of reflexivity due to  James \cite{J} that if  $X$ admits an equivalent $W2R$ norm then $X$ is reflexive. H\'ajek and Johanis proved the converse:   every reflexive Banach space admits an equivalent $W2R$ norm \cite{HJ}.  
 Odell and Schlumprecht \cite{OS} proved that every separable   reflexive Banach space $X$ admits an equivalent $2R$ norm (cf.\,\cite{G}). However,  it is an open question whether every  reflexive Banach space admits an equivalent 
$2R$ norm.

Let $\Gamma$ be an infinite set. Throughout,  $\mathcal{F}$ denotes a collection of finite subsets of $\Gamma$ satisfying the following: \begin{itemize} \item $\mathcal{F}$ contains all singletons;  
\item $\mathcal{F}$ is hereditary, i.e., if $F \in \mathcal{F}$ and $G \subseteq F$ then $G \in \mathcal{F}$;\item $\mathcal{F}$ is compact, i.e., $\{ 1_F \colon F \in \mathcal{F}\}$
is a compact subset of $\{0,1\}^{\Gamma}$ in the topology of pointwise convergence. \end{itemize}
Let $(e_\gamma)_{\gamma \in \Gamma}$ denote the unit vector basis of $c_{00}(\Gamma)$ and let $(e_\gamma^*)$ denote the dual basis.  We define a norm $\|\cdot\|$ on $c_{00}(\Gamma)$ as follows:
\begin{equation} \label{eq: Baernsteinnorm} \|\sum a_\gamma e_\gamma \| = \sup (\sum_{i=1}^n (\sum_{\gamma \in F_i} |a_\gamma|)^2)^{1/2}, \end{equation}
where the supremum is taken over all $n \ge 1$ and all disjoint $F_i \in \mathcal{F}$ ($1 \le i \le n$).

The generalized  Baernstein space $(B(\mathcal{F}),\|\cdot\|)$ is the completion  of $c_{00}(\Gamma)$ with respect to $\|\cdot\|$.  Note that   $(e_\gamma)_{\gamma \in \Gamma}$ is a $1$-unconditional basis of $B(\mathcal{F})$ and that $\|\cdot\|$ satisfies a lower $2$-estimate for disjointly supported vectors $x,y$:
\begin{equation} \label{eq: lower2estimate}  \|x + y\|^2 \ge \|x\|^2 + \|y\|^2. \end{equation} 

The first  space of this type was introduced by Baernstein \cite{B}  with  $\Gamma= \mathbb{N}$ and $\mathcal{F}= \mathcal{S}_1= \{E \subset \mathbb{N}\colon |E| \le \min E\}$ (the Schreier family),    with the extra assumption that $\max F_i < \min F_{i+1}$ for  $1 \le i \le n-1$ in \eqref{eq: Baernsteinnorm}. It was the first example of a reflexive
 Banach space with a  normalized basis  (weakly null by reflexivity) whose arithmetic means do not converge strongly to zero.

The space $B(\mathcal{F})$ is reflexive (for arbitrary $\Gamma$ and $\mathcal{F}$).  For completeness we present a proof at the end of the paper. 

 The norm of  $(B(\mathcal{F},\|\cdot\|)$ and its dual norm $\|\cdot\|_*$  are not $2R$ in general. For example, for the original Baernstein space,  we have $$\|e_n + e_m\| =2,\quad \|e^*_3 +  e_n^*\|_*=1,  \|(e_3^* + e_n^* )+(e_3^* +  e_m^*)\|_* = 2\quad (m,n > 3),$$
and hence $\|\cdot\|$ and $\|\cdot\|_*$ are not $2R$ norms.

 The following question  is open to the best of our knowledge.

\begin{question} Suppose $\Gamma$ is uncountable. Does $B(\mathcal{F})$ have an equivalent $2R$ norm?
\end{question} In Section~\ref{sec: transfinit}, motivated by the Schreier hierarchy introduced in \cite{AA},  we present  a general method for defining, for each countable ordinal $\alpha$,  a family $\mathcal{F}_\alpha$ for certain uncountable $\Gamma$.
The construction is similar to that of  the transfintely defined families introduced  in \cite{AM}.  In Section~\ref{sec: transfinite2R} we prove  that, for each countable ordinal $\alpha$, $B(\mathcal{F}_\alpha)$ has an equivalent $2R$ norm.

In Section~\ref{sec: duals} we prove, for arbitrary $\Gamma$ and $\mathcal{F}$,  that  $B(\mathcal{F})^*$ admits an equivalent $2R$ norm. The renorming is essentially the same as the $W2R$ renorming given in \cite{HJ}.

As an application of these results  we prove that the space constructed by Kutzarova and Troyanski \cite{KT} (based on a family of sets introduced in \cite{BS}) which does not admit an equivalent  norm that is either  uniformly rotund in every direction or uniformly differentiable in every direction   does admit an equivalent $2R$ norm.

In forthcoming articles we prove positive results for other classes of spaces. In particular, in   \cite{DKM}  we consider the existence of equivalent symmetric $2R$ norms for spaces with a symmetric basis.

\section{Transfinitely defined families} \label{sec: transfinit}  \begin{itemize} \item
 Let  $S$ be any set of cardinality at least  $2$ and let $\overline{S} := S^{\mathbb{N}}$.  \label{sec: transfinite}  \item
For distinct $p = (p(i))_{i=1}^\infty \in \overline{S}$ and $q = (q(i))_{i=1}^\infty \in \overline{S}$, let
$d(p,q) = 1$ if $p(1) \ne q(1)$ and, for $k \ge 2 $, let $d(p,q) =  k$ if $p(k) \ne q(k)$ and $p(j) = q(j)$ for $1 \le j \le k-1.$  
\item For $A \subset \overline{S}$, with $|A| \ge 2$, let $$A^\sharp = \min \{d(p,q) \colon p,q \in A, p \ne q\}.$$
We define, for each countable ordinal $\alpha$, a  hereditary family $\mathcal{F}_\alpha$ of finite subsets of  $\overline{S}$.
\item Let $$\mathcal{F}_0  = \{ \emptyset \} \cup \{\{p\} \colon p \in \overline{S}\}.$$
\item  If $k \ge 1$ and  $\mathcal{F}$ is any collection of finite subsets of $\overline{S}$ satisfying the conditions set out in the Introduction, let
$$\mathcal{F}^{(k)} = \mathcal{F}_0 \cup \{A \in \mathcal{F} \colon A^{\sharp} \ge k\}.$$
Note that since  $\mathcal{F}$ is hereditary,  $\mathcal{F}^{(k)}$ is also hereditary.
\item If $\alpha = \beta^+$ is a successor ordinal, let $\mathcal{F}_\alpha$ be any hereditary family satisfying the following:
\begin{itemize} \item $\mathcal{F}_\beta \subseteq \mathcal{F}_\alpha$.
\item If  $A \in \mathcal{F}_\alpha$ and  $|A| \ge 2$, then  there exist  $A_i \in \mathcal{F}_\beta$ ($1 \le i \le A^{\sharp}$) such that
$$ A = \cup_{i=1}^{A^{\sharp}} A_i.$$
\end{itemize}
\item If $\alpha$ is a limit ordinal, choose $\alpha_r \uparrow \alpha$ ($r \ge 1$) and define
$$ \mathcal{F}_\alpha = \cup_{r=1}^\infty \mathcal{F}^{(r)}_{\alpha_r}.$$ Note that, for each $k \ge 1$,
$$\mathcal{F}^{(k)}_\alpha = \cup_{r=1}^\infty \mathcal{F}^{(r \vee k)}_{\alpha_r},$$
where $r \vee k := \max(r,k)$.

\end{itemize}

\section{$B(\mathcal{F}_\alpha^{(k)})$ admits an equivalent  $2R$ norm} \label{sec: transfinite2R}

\begin{theorem} \label{thm: transfinite2R} For each countable ordinal  $\alpha$ and $k \ge 1$, $B(\mathcal{F}_\alpha^{(k)})$ admits a $2R$ renorming. \end{theorem}
We shall use the following  characterization of $2$-rotundity (see e.g., \cite[II.6.4]{DGZ} or \cite{HJ}): $\|\cdot\|$ is a $2R$ norm on $X$  if for all  $(x_n) \subset X$ such that \begin{equation} \label{eq: alternativedef}
\lim_{m,n \rightarrow \infty} [\| x_m + x_n\|^2 - 2(\|x_m\|^2 + \|x_n\|^2)] = 0, \end{equation}
 there exists $x \in X$ such that $x=\lim_{n \rightarrow \infty} x_n$ strongly.

 For $x \in B(\mathcal{F})$, the support of $x$, denoted $\operatorname{supp} x$, is defined by
$$ \operatorname{supp} x = \{\gamma \in \overline{S}: e_\gamma^*(e_\gamma) \ne 0\}.$$

Let $\|\cdot\|_{\alpha,k}$ denote the norm in $B(\mathcal{F}_\alpha^{(k)}).$

\begin{lemma} \label{lem: keyfirstlemma} Let $\alpha$ be a limit ordinal (with $\alpha_r \uparrow \alpha$ as above) and let $k \ge 1$. Suppose that  $\|x_n\|_{\alpha, k} \le 1$ ($n\ge1$)
and that \begin{equation}\label{eq: lim_m_n}
 \lim_{m,n\rightarrow \infty}  \|x_m + x_n\|_{\alpha,k} = 2. \end{equation}   Then, for some $r \ge 1$, $\limsup_{n\rightarrow\infty}\|x_n\|_{\alpha_r, r \vee k} >0$
\end{lemma} \begin{proof} Suppose, to derive a contradiction, that $\lim_{n \rightarrow \infty}\|x_n\|_{\alpha_r, r \vee k} =0$ for all $r \ge 1$;  in particular, 
$x_n \rightarrow 0$ in $\ell_2(\overline{S})$. Hence, by a gliding hump argument, approximating by finitely disjointly  supported vectors, and after passing to a subsequence and relabelling, we may assume that $ \operatorname{supp} x_n$ is finite  and that $ \operatorname{supp} x_n \cap  \operatorname{supp} x_m = \emptyset$ if $m \ne n$.

Fix $n \ge 1$ and $F \in \mathcal{F}^{(k)}_\alpha$  satisfying
$$ |F \cap  \operatorname{supp} x_n| \ge 2.$$ Let
$$N = \max \{ d(p,q) \colon p,q \in  \operatorname{supp} x_n, p \ne q\}.$$ It follows that $F^\sharp \le N$, and hence
$$F \in \cup_{r=1}^N \mathcal{F}_{\alpha_r}^{(r \vee k)}.$$

Let $$x_n = \sum a_\gamma e_\gamma$$ and, for $m > n$, 
$$x_m = \sum b^m_\gamma e_\gamma.$$
Since $x_n \rightarrow 0$ in $\ell_2(\overline{S})$,
\begin{equation} \label{eq: sumsofsquares}
\lim_{n \rightarrow \infty}  \sum a_\gamma^2 =0 .\end{equation}
Since the supports of the $x_m$'s are disjoint, we may assume that $a_\gamma\ge0$ and $b^m_\gamma\ge0$. 

By assumption, $\|x_m\|_{\alpha_r, r \vee k} \rightarrow 0$  as $m \rightarrow \infty$ for all $r \ge 1.$  Hence  \begin{equation} \label{eq: uniform1}
\lim_{m \rightarrow \infty} \sum_{\gamma \in  F} b^m_\gamma = 0. \end{equation}
\textit{uniformly} over all $F \in \mathcal{F}^{(k)}_\alpha$ satisfying $|F \cap \operatorname{supp} x_n| \ge 2$.

Note that if $F_1,F_2,\dots,F_s$ are \textit{disjoint}  sets in $\mathcal{F}^{(k)}_\alpha$ satisfying $|F_i \cap \operatorname{supp} x_n| \ge 2$  ($1 \le i \le s$)
then $s \le |\operatorname{supp} x_n|$. Hence \eqref{eq: uniform1} implies that
\begin{equation} \label{eq: uniform2} \sum_{\gamma \in \cup_{i=1}^s F_i} b^m_\gamma \rightarrow 0 \end{equation}
as $m \rightarrow \infty$ \textit{uniformly} over all such collections $(F_i)_{i=1}^s$.  Let $A_i = \sum_{\gamma \in F_i} a_\gamma$
and let $B^m_i =  \sum_{\gamma \in  F_i} b^m_\gamma$. Then \begin{align*}
\sum_{i=1}^s (A_i + B^m_i)^2 &= \sum_{i=1}^s (A_i^2 +( B^m_i)^2 + 2A_iB^m_i) \\
&\le  \sum_{i=1}^s A_i^2 + (\sum_{i=1}^s B^m_i)^2 + 2(\sum_{i=1}^s B^m_i) (\sum A_i^2)^{1/2}\\
&\le \sum_{i=1}^s A_i^2 +  (\sum_{i=1}^s B^m_i)^2 + 2\|x_n\|_{\alpha,k} \sum_{i=1}^s B^m_i\\
& \le \sum_{i=1}^s A_i^2 +  (\sum_{i=1}^s B^m_i)^2 + 2 \sum_{i=1}^s B^m_i.
\end{align*}   Note that  \eqref{eq: uniform2} implies that $\sum_{i=1}^s B^m_i \rightarrow 0$ as $m \rightarrow \infty$ uniformly over all such $(F_i)_{i=1}^s$.
Let $\varepsilon>0$.  It follows that  for all $m \ge M(n, \varepsilon)$,   \begin{equation}  \label{eq: atleast2}
\sum_{i=1}^s (A_i + B^m_i)^2 < \sum_{i=1}^s A_i^2 + \varepsilon \le \|x_n\|_{\alpha_k} + \varepsilon \le 1 + \varepsilon\end{equation} uniformly over all $(F_i)_{i=1}^s$.
Moreover, it follows  from \eqref{eq: sumsofsquares}  that for all $n \ge N(\varepsilon)$   \begin{equation*}
 \sum a_\gamma^2  < \varepsilon^2. \end{equation*}
Let $J\subset \operatorname{supp} x_n$. Consider disjoint sets $G_\lambda  \in \mathcal{F}^{(k)}_\alpha$ ($\lambda \in J$) satisfying $G_\lambda \cap \operatorname{supp} x_n = \{\lambda\}$ ($\lambda \in J$). Let $C^m_\lambda =  \sum_{\gamma \in  G_\lambda} b^m_\gamma$.  Then
for all $m> n> N(\varepsilon)$,
\begin{equation} \label{eq: exactlyone} \begin{split} \sum_{\lambda \in J} (a_\lambda + C^m_\lambda)^2 &\le \sum_{\lambda \in J} a_\gamma^2 
+ \sum_{\lambda \in J}(C^m_\lambda)^2
+ 2( \sum_{\lambda \in J} a_\lambda^2)^{1/2} (\sum_{\lambda\in J}(C^m_\lambda)^2)^{1/2}\\
&\le \varepsilon +  \sum_{\lambda \in J}(C^m_\lambda)^2 + 2\varepsilon \|x_m\|_{\alpha,k}\\
&\le \varepsilon + 2\varepsilon +  \|x_m\|_{\alpha,k}^2 \\ &\le 1 + 3\varepsilon. \end{split}
\end{equation}
Hence, combining  \eqref{eq: atleast2} and \eqref{eq: exactlyone}, for all $n \ge N(\varepsilon)$ and $m> M(n,\varepsilon)$,
\begin{equation} \|x_n + x_m\|_{\alpha,k}^2 \le  2 + 4\varepsilon.
\end{equation}  Since $\varepsilon>0$ is arbitrary, we have  \begin{equation} \label{eq: sqrt2}
\limsup_{n \rightarrow \infty} \limsup_{m \rightarrow \infty} \|x_n + x_m\|_{\alpha,k}  \le \sqrt{2},
\end{equation} which contradicts \eqref{eq: lim_m_n}. 
\end{proof} The following analogue for successor ordinals has a similar   (but simpler) proof.
\begin{lemma}\label{lem: successorcaseoflemma}  Let $\alpha = \beta^+$ be a successor ordinal.  Suppose that  $\|x_n\|_{\alpha, k} \le 1$ ($n\ge1$) and that \begin{equation}\label{eq: lim_m_n_2_successor}
 \lim_{m,n\rightarrow \infty}  \|x_m + x_n\|_{\alpha,k} = 2.
 \end{equation}   Then $$ \limsup_{n \rightarrow \infty} \|x_n\|_{\beta,k} >0.$$
\end{lemma}
 \begin{remark}\label{rem: improvementoflemma} \eqref{eq: sqrt2}   shows that  Lemma~\ref{lem: keyfirstlemma} and Lemma~\ref{lem: successorcaseoflemma} can be strengthened  by replacing \eqref{eq: lim_m_n}  and \eqref{eq: lim_m_n_2_successor} by
$$\limsup_{n \rightarrow \infty} \limsup_{m \rightarrow \infty} \|x_n + x_m\|_{\alpha,k}  >  \sqrt{2}.$$
\end{remark}

The proof of the following lemma uses the fact that  Hilbert space $(\ell_2,|\cdot|)$ is uniformly convex; specifically,  for $0< \varepsilon < 2$,
\begin{equation} |x| \le1, |y| \le 1, |x-y| = \varepsilon \Rightarrow |\frac{x+y}{2}| \le 1 - \frac{\varepsilon^2}{8}. \end{equation}
We will also use the following notation: for $x =\sum_{\gamma \in \overline{S}} x_\gamma e_\gamma$ and disjoint sets  $F_i \subset \overline{S}$ ($1 \le i \le n$),
$$|(x; F_1,\dots,F_n)|_2:=(\sum_{i=1}^n (\sum_{\gamma \in F_i} x_\gamma)^2)^{1/2}.$$
 Note that if $x \ge 0$,  then
$$\|x\|_{\alpha,k} = \sup |(x;
 F_1,\dots,F_n)|_2,$$
where the supremum is taken over all $n \ge 1$ and disjoint $F_i \in \mathcal{F}^{(k)}_\alpha$.

 \begin{lemma}  \label{lem: limitordinalkeylemma}  Let $\alpha$ be a limit ordinal (with $\alpha_r \uparrow \alpha$ as above) and let $k \ge 1$. Suppose that  $\|x_n\|_{\alpha, k} \le 1$ ($n\ge1$), that \begin{equation}\label{eq: lim_m_n_2}
 \lim_{m,n\rightarrow \infty}  \|x_m + x_n\|_{\alpha,k} = 2,
 \end{equation}    and that  there exists $x \in \ell_2(\overline{S})$ such that, 
 for each  $r \ge 1$,
\begin{equation}\lim_{n \rightarrow \infty} \|x_n - x\|_{\alpha_r,r \vee k} = 0.
\end{equation} Then  $\lim_{n \rightarrow \infty} \|x_n - x\|_{\alpha,k}=0$.
\end{lemma} \begin{proof} Note that 
$$\|x\|_{\alpha,k} \le \limsup_{n \rightarrow \infty} \|x_n\|_{\alpha,k}  \le 1,$$
since $x_n \rightarrow x$ pointwise.
 Suppose, to derive a contradiction, that the conclusion is false. Then, after passing to a subsequence and relabelling, we may assume that $$\lim_{n \rightarrow\infty}\|x_n - x\|_{\alpha,k} = \delta > 0.$$
Let $x_n^\prime=  x_n - x$. By assumption, for all $r \ge 1$,
 $$\lim_{n \rightarrow \infty}\|x_n^\prime\|_{\alpha_r, r\vee k} = 0.$$ 
Let $\varepsilon>0$. Choose a finitely supported  vector $y$ such that \begin{equation*} \|x - y\|_{\alpha,k} < \frac{\varepsilon^2}{10}. \end{equation*}
By a gliding hump argument, passing to a further subsequence and relabelling, we may choose disjointly supported  vectors $y_n$ ($n \ge 1$), each  with finite support disjoint from the support of $y$, such that $\|x_n^\prime - y_n\|_{\alpha,k} \rightarrow 0$ as $n \rightarrow \infty$ and, for all $m,n \ge 1$,
\begin{equation*} \|y + y_n||_{\alpha,k} \le 1, \end{equation*} and also
\begin{equation*} \|2y + y_n + y_m \|_{\alpha,k} > 2  -\frac{\varepsilon^2}{4}. \end{equation*}
Hence $$\lim_{n \rightarrow \infty}\| y_n\|_{\alpha,k} = \delta,$$
and, for all $r \ge 1$, \begin{equation} \label{eq:  nullsequence}
 \lim_{n \rightarrow \infty}\| y_n\|_{\alpha_r, r \vee k} =\lim_{n \rightarrow \infty}\|x_n^\prime\|_{\alpha_r, r\vee k} =0. \end{equation}
Without loss of generality, we may assume that $y\ge0$ and $y_n \ge 0$ for all $n \ge 1$.
Fix $n \ge 1$ and let $m>n$. Suppose that $2y + y_n + y_m$ is normed by disjoint sets $F_1,\dots,F_u$ in $\mathcal{F}_{\alpha,k}$ (we suppress the dependence of $F_i$ on $n$ and $m$ to simplify notation), i.e.,
$$|(2y + y_n + y_m; F_1,\dots,F_u)|_2 =  \|2y + y_n + y_m \|_{\alpha,k}> 2  -\frac{\varepsilon^2}{4}. $$
Since $$|(y+y_n; F_1,\dots,F_u)|_2 \le \|y+ y_n\|_{\alpha,k}\le 1$$ and $$|(y+y_m; F_1,\dots,F_u)|_2 \le \|y+ y_m\|_{\alpha,k}\le 1,$$
the uniform convexity of $\ell_2$ yields
$$|(y_n-y_m; F_1,\dots,F_u)|_2 < \varepsilon.$$

 We may assume that $F_1,\dots,F_s$ have nonempty intersection with \textit{both} $\operatorname{supp} y$ and   $\operatorname{supp} y_n$, that $F_{s+1},\dots,F_t$ intersect  $\operatorname{supp} y$ but \textit{not}
 $\operatorname{supp} y_n$, and that $F_{t+1},\dots,F_u$ do \textit{not} intersect   $\operatorname{supp} y$. Note that
$s \le |\operatorname{supp} y|$  and $|F_i \cap \operatorname{supp} (y + y_n)| \ge 2$ for $1 \le i \le s$. Hence, repeating the argument  used to prove \eqref{eq: uniform2}, we deduce that
\begin{equation} \label{eq: uniform3} \lim_{m\rightarrow\infty}  \sum_{\gamma \in \cup_{i=1}^s F_i} b^m_\gamma = 0 \end{equation}
 for $y_m = \sum b^m_\gamma e_\gamma$. Hence
$$|(y_m; F_1,\dots,F_s)|_2 < \frac{\varepsilon}{2}$$
for all  $m > M_1(n,\varepsilon)$.

Note that $y_n$ vanishes  on  $F_i$ for $s+1 \le i \le t$.
 Hence, for all $m > M_1(n,\varepsilon)$,  \begin{align*}
|(y_n+y_m; F_1,\dots,F_t)|_2  &=(\sum_{i=1}^s (\sum_{\gamma \in F_i} (b^n_\gamma + b^m_\gamma))^2 + \sum_{i = s+1}^t (\sum_{\gamma \in F_i} b^m_\gamma)^2)^{1/2} \\
&\le (\sum_{i=1}^s (\sum_{\gamma \in F_i} (b^n_\gamma -b^m_\gamma))^2 + \sum_{i = s+1}^t (\sum_{\gamma \in F_i} b^m_\gamma)^2)^{1/2} \\
& + 2 (\sum_{i=1}^s (\sum_{\gamma \in F_i} b^m_\gamma)^2)^{1/2}\\
 \intertext{(by the triangle inequality in $\ell_2$)}
&= |(y_n-y_m; F_1,\dots,F_t)|_2 + 2|(y_m; F_1,\dots F_s)|_2\\
&\le \varepsilon + \varepsilon = 2\varepsilon.
\end{align*}
So \begin{align*}
|(2y + y_n + y_m; F_1,\dots,F_t)|_2 &\le  2|(y;F_1,\dots, F_t)|_2 + |(y_n+y_m;F_1,\dots,F_t)|_2\\ &\le 2\|y\|_{\alpha,k} + 2\varepsilon. \end{align*}
Thus, \begin{equation} \label{eq: firstest}\begin{split}
(2 - \frac{\varepsilon^2}{4})^2 &< |(2y+y_n+y_m; F_1,\dots,F_u)|_2^2\\
&=  |(2y+y_n+y_m; F_1,\dots,F_t)|_2^2 +  |(y_n+y_m; F_{t+1},\dots,F_u)|_2^2\\
\intertext{(since $y$ vanishes on $F_i$ for $t+1 \le i \le u$)}
&\le ( 2\|y\|_{\alpha,k} + 2\varepsilon)^2 + \|y_n + y_m\|_{\alpha,k}^2.
\end{split}
\end{equation} Since $y$, $y_n$, and $y_m$ are disjointly supported, we have \begin{equation} \label{eq: secondest}
\|y\|_{\alpha,k}^2 + \|y_n\|_{\alpha,k}^2  \le \|y+y_n\|_{\alpha,k}^2  \le 1 \end{equation}
and \begin{equation} \label{eq: thirdest}
\|y\|_{\alpha,k}^2 + \|y_m\|_{\alpha,k}^2  \le \|y+y_m\|_{\alpha,k}^2  \le 1. \end{equation}
Combining \eqref{eq: firstest}, \eqref{eq: secondest}, and \eqref{eq: thirdest}, \begin{equation*} \begin{split}
4\|y\|_{\alpha,k}^2 +2(\|y_n\|_{\alpha,k}^2 + \|y_m\|_{\alpha,k}^2) &\le 4\\
& = (2 - \frac{\varepsilon^2}{4})^2 + \varepsilon^2 -\frac{\varepsilon^4}{16}\\
&\le  ( 2\|y\|_{\alpha,k} + 2\varepsilon)^2 + \|y_n + y_m\|_{\alpha,k}^2 + \varepsilon^2 \\
&\le 4\|y\|_{\alpha,k}^2 + \|y_n + y_m\|_{\alpha,k}^2  +(8\varepsilon +  5\varepsilon^2).
\end{split} 
\end{equation*} Hence for all $m>M_1(n,\varepsilon)$,  \begin{equation} \label{eq: 8epsilon}
\|y_n + y_m\|_{\alpha,k}^2 +(8\varepsilon +  5\varepsilon^2) \ge 2(\|y_n\|_{\alpha,k}^2 + \|y_m\|_{\alpha,k}^2).  \end{equation}
Now suppose $\varepsilon$ is chosen so that   $8\varepsilon +  5\varepsilon^2 < 2\delta^2$. 
Since $\lim_{n \rightarrow \infty} \|y_n\|_{\alpha,k} = \delta$, it follows from \eqref{eq: 8epsilon} that
$$\liminf_{n \rightarrow \infty} \liminf_{m \rightarrow\infty} \|y_n + y_m\|_{\alpha,k}  > \sqrt2 \delta.$$
which contradicts Remark~\ref{rem: improvementoflemma}  since, for all $r \ge 1$,
$$\lim_{n \rightarrow \infty} \|y_n\|_{\alpha_r, r \vee k} = 0.$$
\end{proof}  The following analogue for successor ordinals has a similar   (but simpler) proof.
\begin{lemma}  \label{lem: successorordinalkeylemma} Let $\alpha = \beta^+$ be a successor ordinal.  Suppose that  $\|x_n\|_{\alpha, k} \le 1$ ($n\ge1$), that \begin{equation}\label{eq: lim_m_n_2}
 \lim_{m,n\rightarrow \infty}  \|x_m + x_n\|_{\alpha,k} = 2,
 \end{equation}    and that  there exists $x \in \ell_2(\overline{S})$ such that 
\begin{equation}\lim_{n \rightarrow \infty} \|x_n - x\|_{\beta, k} = 0.
\end{equation} Then  $\lim_{n \rightarrow \infty} \|x_n - x\|_{\alpha,k}=0$.
\end{lemma}
\begin{proof}[Proof of Theorem~\ref{thm: transfinite2R}]  We will prove the result for a fixed $\alpha$ and for all $k \ge 1$ by transfinite induction  on $\alpha$. The result clearly holds for $\alpha = 0$ since  $B(\mathcal{F}_0^{(k)}) = B(\mathcal{F}_0)=  \ell_2(\overline{S})$ for all $k \ge 1$.
So suppose the result holds for all $\beta<\alpha$ and for all $k \ge 1$.

\textbf{Case I}: $\alpha$ is a limit ordinal. So
$\mathcal{F}^{(k)}_\alpha = \cup_{r=1}^\infty \mathcal{F}^{(r \vee k)}_{\alpha_r},$ where $\alpha_r \uparrow \alpha$.  By inductive hypothesis, each 
$B(\mathcal{F}^{(r \vee k)}_{\alpha_r})$ admits an equivalent $2R$ norm $ |\!|\!|\cdot |\!|\!|_{\alpha_r, r \vee k}$. Note that 
$$ |\!|\!|\cdot  |\!|\!|_{\alpha_r, r \vee k} \le C_r \|\cdot\|_{\alpha, k}$$ for some $C_r < \infty$. Thus,
$$  |\!|\!|\cdot |\!|\!|_{\alpha,k}^2 := \|\cdot\|_{\alpha,k}^2 + \sum_{r=1}^\infty \frac{1}{2^rC_r^2} |\!|\!|\cdot |\!|\!|_{\alpha_r, r \vee k}^2$$
defines an equivalent norm $ |\!|\!|\cdot |\!|\!|_{\alpha,k}$ on $B(\mathcal{F}_\alpha^{(k)})$. Let us show that $ |\!|\!|\cdot |\!|\!|_{\alpha,k}$ is a $2R$ norm.
Suppose that $(x_n) \subset B(\mathcal{F}_\alpha^{(k)})$ satisfies
$$ \lim_{m,n \rightarrow \infty}  |\!|\!|x_n + x_m |\!|\!|_{\alpha,k}^2 - 2( |\!|\!|x_n |\!|\!|_{\alpha,k}^2 +  |\!|\!|x_m| |\!|\!|_{\alpha,k}^2) = 0.$$
Note that \begin{align*}  
& |\!|\!|x_n + x_m |\!|\!|_{\alpha,k}^2 - 2( |\!|\!|x_n |\!|\!|_{\alpha,k}^2 +  |\!|\!|x_m |\!|\!|_{\alpha,k}^2)\\ &\le -(\|x_n\|_{\alpha,k} -\| x_m\|_{\alpha,k})^2 - \sum_{r=1}^\infty \frac{1}{2^rC_r^2} ( |\!|\!|x_n |\!|\!|_{\alpha_r,k} - |\!|\!| x_m |\!|\!|_{\alpha+r,k})^2.
\end{align*} It follows that $\lim_{n \rightarrow \infty} \|x_n\|_{\alpha,k} = L$ for some $L \ge 0$, that \begin{equation} \label{eq: limitalphak}
\lim_{m,n \rightarrow \infty} \|x_n + x_m\|_{\alpha,k}^2 - 2(\|x_n\|_{\alpha,k}^2 + \|x_m\|_{\alpha,k}^2) = 0, \end{equation}
and that,  for all $r\ge1$,
 \begin{equation*}
\lim_{m,n \rightarrow \infty}  |\!|\!|x_n + x_m |\!|\!|_{\alpha_r,k \vee r}^2 - 2( |\!|\!|x_n |\!|\!|_{\alpha_r,k \vee r}^2 +  |\!|\!|x_m |\!|\!|_{\alpha_r,k \vee r}^2) = 0. \end{equation*}
Since each $ |\!|\!|\cdot  |\!|\!|_{\alpha_r, r \vee k}$ is a $2R$ norm, it follows from \eqref{eq: alternativedef} that  there exists $x \in \ell_2(\overline{S})$ such that, for all $r\ge 1$,
$$ \lim_{n \rightarrow \infty}  |\!|\!|x_n - x |\!|\!|_{\alpha_r, r \vee k} = 0.$$
 Moreover, \eqref{eq: limitalphak} implies that
$$ \lim_{m,n \rightarrow \infty}\| x_n + x_m \|_{\alpha,k} = 2L.$$ So, by Lemma~\ref{lem: limitordinalkeylemma},
$$\lim_{n \rightarrow \infty} \| x_n - x \|_{\alpha,k} = 0,$$
and hence $$\lim_{n \rightarrow \infty}  |\!|\!| x_n - x  |\!|\!|_{\alpha,k} = 0,$$
as desired. 

\textbf{Case II}: $\alpha= \beta^+$ is a successor ordinal.
The proof is very similar to the limit ordinal case. By the inductive hypothesis, $B(\mathcal{F}_\beta^{(k)})$ admits an equivalent $2R$ norm  $ |\!|\!| \cdot   |\!|\!|_{\beta,k}$. Let
$$  |\!|\!| \cdot |\!|\!|^2_{\alpha, k}= \|\cdot\|_{\alpha,k}^2  +  |\!|\!| \cdot |\!|\!|^2_{\beta,k}.$$
Using  Lemma~\ref{lem: successorordinalkeylemma} instead of  Lemma~\ref{lem: limitordinalkeylemma} and repeating the  argument  of  Case~I shows that    $|\!|\!| \cdot |\!|\!|_{\alpha, k}$ is a  $2R$ norm.
\end{proof} 

\section{$B(\mathcal{F})^*$ admits  an equivalent  $2R$ norm} \label{sec: duals}
Let $\mathcal{F}$ be a compact, hereditary family of finite subsets of an infinite set $\Gamma$ containing all singleton sets. We prove in Section~\ref{sec: reflexivity} that  $(B(\mathcal{F}),\|\cdot\|)$ is reflexive.
Day \cite{D}  introduced the  norm $\|\cdot \|_{\textrm{Day}}$ on $c_0(\Gamma)$  defined by 
$$\|\sum a_\gamma e_\gamma\|_{\textrm{Day}} = \sup (\sum_{i=1}^n 4^{-i} |a_{\gamma_i}|^2)^{1/2},$$
where the supremum is taken over all $n \ge 1$ and all  choices of distinct $\gamma_i\in \Gamma$ ($1 \le i \le n$). We define an equivalent norm
on   $B(\mathcal{F})^*$ thus:
$$  |\!|\!|x \  |\!|\!|^2 = \|x\|_*^2 +  \|x \|_{\textrm{Day}}^2 \qquad (x \in B(\mathcal{F})^*).  $$

 The following result is essentially  due to H\'ajek and Johannis.  It is a consequence of Theorem~3 and Corollary~4 of \cite{HJ} and the reflexivity of $B(\mathcal{F})^*$.
\begin{lemma}  \label{lem: HJresult} Suppose $(y_n) \subset B(\mathcal{F})^*$ satisfies \begin{equation} \label{eq: triplenormcondition}
 \lim_{m,n \rightarrow \infty} |\!|\!|y_n + y_m |\!|\!|^2- 2(|\!|\!|y_n|\!|\!|^2 + |\!|\!|y_m|\!|\!|^2) = 0. \end{equation}
Then there exists $y \in B(\mathcal{F})^*$ such that \begin{equation*}
 y_n\rightarrow y\quad \text{weakly as $n \rightarrow \infty$}\end{equation*} and
$$ \lim_{n \rightarrow \infty} \|y_n - y\|_\infty = 0.$$
\end{lemma} Dualizing \eqref{eq: lower2estimate},  the dual space $(B(\mathcal{F})^*,\|\cdot\|_*)$
satisfies an upper $2$-estimate for disjointly supported vectorrs $x,y \in B(\mathcal{F})^*$: 
$$\|x + y\|_*^2 \le \|x\|_*^2 + \|y\|_*^2.$$ Moreover, for all $x \in B(\mathcal{F})^*$ and $F \in \mathcal{F}$,
$$\|x\cdot 1_F\|_* = \|x\cdot 1_F\|_\infty \le \|x\|_\infty.$$ \begin{lemma} \label{lem: normoflimit}  Suppse that  $y$ and $y_n$ have disjoint finite supports ($n \ge 1$), that $$\|y\|_* = \|y_n\|_* =1 \qquad(n \ge 1),$$
and that $$\lim_{n \rightarrow \infty} \|y_n\|_\infty = 0.$$
Then, for all $\delta > 0$,
$$ \lim_{n \rightarrow \infty} \|y+\delta y_n\|_* = (1 + \delta^2)^{1/2}.$$
\end{lemma} \begin{proof} We may assume that $y \ge0$ and $y_n \ge.$ Choose positive norming vectors $x, x_n \in B(\mathcal{F})$ with
$$(x,y) = \|x\| = \|y\|_* =1, \quad  (x_n,y_n)=\|x_n\|= \|y_n\|_*=1,$$ 
where $(\cdot,\cdot)$ denotes the duality pairing for $B(\mathcal{F})\times B(\mathcal{F})^*$. Note that $x$ and $x_n$ have disjoint finite supports ($n \ge 1$). Fixing $n \ge 1$,  choose disjoint  $F_i \in \mathcal{F}$ ($1 \le i \le N$) such that
$$\|x + \delta x_n\| = |(x+ \delta x_n; F_1,\dots,F_N)|_2.$$
We may assume that only $F_1,\dots,F_k$ have non-empty intersection with both $\operatorname{supp} x$ and $\operatorname{supp} x_n$.
Note that $$k \le M:= |\operatorname{supp} x|.$$ For each $1 \le i \le k$,
$$\|y_n \cdot 1_{F_i}\|_* \le \|y_n\|_\infty.$$ Hence
$$\sum_{i=1}^k \|y_n \cdot 1_{F_i}\|_* \le M \|y_n\|_\infty  \rightarrow 0 \quad\text{as $n \rightarrow \infty$}.$$
Let $F = \cup_{i=1}^k F_i$. (To simplify notation we  suppress the dependence of $F$ on $n$.) Then \begin{align*}
\|x_n - x_n \cdot 1_F\| &\ge (x_n -x_n\cdot 1_F, y_n)\\ &= (x_n,y_n -y_n\cdot 1_F)\\
&= 1 - (x_n, y_n\cdot 1_F)\\ &\ge 1 - \sum_{i=1}^k \|y_n \cdot 1_{F_i}\|_*\\
&\rightarrow 1 \quad \text{as $n \rightarrow \infty$.}
\end{align*}Hence $$\lim_{n \rightarrow \infty} (\|x_n\|^2 - \|x_n - x_n\cdot 1_F\|^2)= 1 - \lim_{n \rightarrow \infty}  \|x_n - x_n\cdot 1_F\|^2=0.$$
Since $(B(\mathcal{F},\|\cdot\|)$ satisfies a lower $2$-estimate, it follows that $\|x_n \cdot 1_F\| \rightarrow 0$ as $n \rightarrow \infty$.
Hence \begin{align*} 1  + \delta^2 &\le \liminf_{n \rightarrow \infty}\|x + \delta x_n\|^2\\
& \le  \limsup_{n \rightarrow \infty}\|x + \delta x_n\|^2\\
& = \limsup_{n \rightarrow \infty} |(x+\delta x_n; F_1,\dots F_N)|_2^2\\
&= \limsup_{n \rightarrow \infty}  |(x+ \delta x_n - \delta x_n \cdot 1_F; F_1,\dots,F_N)|_2^2\\
& \le   \limsup_{n \rightarrow \infty}[ \|x\|^2+ \delta^2 \| x_n - x_n \cdot 1_F\|^2]\\
\intertext{(since no $F_i$  intersects both $\operatorname{supp} x$ and $\operatorname{supp} (x_n - x_n \cdot 1_F)$)}
&= 1 + \delta^2.
\end{align*}  Thus, $$\lim_{n \rightarrow \infty}\|x+ \delta x_n\| ^2= 1 + \delta^2,$$
and hence \begin{align*}
(1 + \delta^2)^{1/2}  &\ge \limsup_{n \rightarrow \infty}  \|y + \delta y_n\|_*\\
\intertext{(since $\|\cdot\|_*$ satisfies an upper $2$-estimate)}
&\ge  \liminf_{n \rightarrow \infty}  \|y + \delta y_n\|_*\\
& \ge  \liminf_{n \rightarrow \infty} \frac{(x + \delta x_n, y + \delta y_n)}{\|x + \delta x_n\|}\\
&=  \frac{1 + \delta^2}{(1 + \delta^2)^{1/2}}\\ &=(1 + \delta^2)^{1/2}.
\end{align*} \end{proof}
\begin{theorem}\label{thm: dualrenorming} $|\!|\!|\cdot  |\!|\!|$ is an equivalent  $2R$ norm for  $B(\mathcal{F})^*$.
\end{theorem} \begin{proof} Suppose $(y_n) \subset B(\mathcal{F})^*$ satisfies \eqref{eq: triplenormcondition}.
By Lemma~\ref{lem: HJresult} there exists $y \in B(\mathcal{F})^*$ such that $y = w-\lim_{n \rightarrow \infty} y_n$
and $\lim_{n \rightarrow \infty}\|y_n - y\|_\infty =0$.  Suppose, to derive a contradiction, that $(y_n)$ does not converge strongly to $y$. 
Passing to a subsequence and relabelling, we may assume that $y_n = y+z_n$ where
$$\lim_{n \rightarrow\infty}  \|z_n\|_* = \delta>0, \lim_{n \rightarrow\infty}  \|z_n\|_\infty = 0,$$
and that  the following limits exist:
$$\lim_{n \rightarrow\infty}  \|y+z_n\|_*, \lim_{n,m \rightarrow\infty}  \|2y+z_n+z_m\|_* .$$
Let $\varepsilon>0$. By passing to a further subsequence, a gliding hump argument and the fact that $ \lim_{n \rightarrow\infty}  \|z_n\|_\infty = 0$
show that there exist vectors $y^\prime$ and $z_n^\prime$ ($n \ge 1$) with disjoint finite supports such that 
\begin{equation} \label{eq: approxi} \|y - y^\prime\|_* < \varepsilon, \lim_{n \rightarrow \infty}\|  z_n - z_n^\prime\|_* =0. \end{equation}
Note that \eqref{eq: triplenormcondition} implies that \begin{equation} \label{eq: limitnorm*}
 \lim_{m,n \rightarrow \infty} \|2y+ z_n + z_m \|_*^2- 2(\|y+z_n|\|_*^2 + \|y+z_m\|_*^2) = 0.  \end{equation}
Since $\lim_{n \rightarrow \infty} \|z^\prime_n\|_\infty = 0$,  Lemma~\ref{lem: normoflimit} yields
$$\lim_{n \rightarrow \infty} \|y^\prime + z_n^\prime\|_*^2 = \|y^\prime\|_*^2 +
 \lim_{n \rightarrow \infty}\|z_n^\prime\|_*^2  = \|y^\prime\|_*^2 + \delta^2.$$
Since $(B(\mathcal{F},\|\cdot\|_*)$ satisfies an upper $2$-estimate,
$$\lim_{n,m \rightarrow \infty} \|2y^\prime + z_n^\prime + z_m^\prime \|_*^2 \le 4\|y^\prime\|_*^2 + 2\delta^2. $$
Hence \begin{align*}
&\limsup_{n,m \rightarrow \infty}[\|2y^\prime + z_n^\prime + z_m^\prime \|_*^2 - 2(\|y^\prime + z_n^\prime\|_*^2 + \|y^\prime + z_m^\prime\|_*^2)]\\ &\le
 4\|y^\prime\|_*^2 + 2\delta^2 - 2(2\|y^\prime\|_*^2 + 2\delta^2)\\
&= -2\delta^2,
\end{align*} which contradicts \eqref{eq: approxi} and \eqref{eq: limitnorm*} provided $\varepsilon$ is sufficiently small.
\end{proof}
 Based on a family of sets introduced in \cite{BS},  Kutzarova and Troyanski \cite {KT}  constructed a Banach space $Y$ which  does not admit an equivalent  norm that is uniformly rotund or uniformly differentiable  in every direction. As an  application of our results, we show   that $Y$ does admit an equivalent $2R$ norm.
  \begin{corollary}  The Banach space $Y$ defined in \cite{KT} admits an equivalent $2R$ norm.
\end{corollary}
\begin{proof}
The space $Y$ is defined as  $X \oplus X^*$, where $X$ is defined below.

 Let $S = \mathbb{N}$. Let $\mathcal{F}_1$ be the collection of all finite subsets  $F$ of $\overline{S}$ such that, if $|F| \ge 2$,  then for all $p \in F$,  $p(1)=1$ and   $p(i) \in \{1,2,\dots,i-1\}$ for all $i \ge 2$ and 
such that  for all distinct $p,q \in F$, there exists $m \ge 3$ such  that $p(i) = q(i)$ for all $1 \le i \le m-1$ and $p(m) \ne q(m)$, which implies that $F^\sharp = m-1$ and $|F| \le m-1$ as required.  The space $X$ is defined to be the closed linear span  of  $$\{e_p \colon p(1)=1, p(i) \in \{1,2,\dots,i-1\},  i \ge 2\}.$$ 
in $B(\mathcal{F}_1)$.
The successor case of the proof of Theorem~\ref{thm: transfinite2R} shows that 
$$ |\!|\!| \cdot |\!|\!|^2  = \|\cdot\|_1^2 + \|\cdot\|^2_{\ell_2(\overline{S})}$$
is an equivalent $2R$ norm on $B(\mathcal{F}_1)$. Hence $|\!|\!| \cdot |\!|\!|$ restricts to an equivalent $2R$ norm on $X$.
By Theorem~\ref{thm: dualrenorming}, $B(\mathcal{F}_1)^*$ admits an equivalent $2R$ norm. Note that $X^*$ is isomorphic to a quotient space of $B(\mathcal{F}_1)^*$. It is easily seen that a quotient norm of a $2R$ norm is $2R$. Hence $X^*$ admits an equivalent $2R$ norm,  $|\!|\!| \cdot |\!|\!|^\prime$ say. Finally,
$$\|(x,x^*)\| =\sqrt{|\!|\!|x|\!|\!|^2 + |\!|\!| x^* |\!|\!|^{\prime2}}\qquad ( (x,x^*) \in X \oplus X^*)$$
is an equivalent $2R$ norm on  $X \oplus X^* =Y$.
\end{proof}
 \section{$B(\mathcal{F})$ is reflexive} \label{sec: reflexivity}
\begin{theorem}  For arbitrary $\Gamma$ and $\mathcal{F}$,  $B(\mathcal{F})$ is reflexive.
\end{theorem} \begin{proof} First, we consider the case $\Gamma = \mathbb{N}$. Let $(e_n)$ denote the unit vector basis of $B(\mathcal{F})$.

  Let $(F_i)_{i=1}^\infty \subset \mathcal{F}$  be a collection of disjoint elements of $\mathcal{F}$ and suppose that $\sum_{i=1}^\infty |a_i|^2 \le 1$. For $x = \sum_{i=1}^\infty x_i e_i  \in B(\mathcal{F})$,
\begin{equation}\label{eq: functionalrep} | \sum_{i=1}^\infty a_i(\sum_{j \in F_i} x_j)| \le  (\sum_{i=1}^\infty |a_i|^2)^{1/2} \|x\| \le \|x\|. \end{equation} Hence we  may identify   $\sum_{i=1}^\infty a_i 1_{F_i} \in \ell_\infty$  with  the element $k$ in the unit ball of $B(\mathcal{F})^*$ defined by \eqref{eq: functionalrep}.

Suppose $x \in B(\mathcal{F})$ has finite support  and that  $$\|x\| = (\sum_{i=1}^n (\sum_{j \in G_i} |x_j|) ^2)^{1/2}$$
for disjoint $G_i \in \mathcal{F}$. There exist  nonnegative  $a_1,\dots,a_n$ with $\sum_{i=1}^n a_i^2 = 1$  such that
$$  \sum_{i=1}^n a_i (\sum_{j \in G_i} |x_j|) = \|x\|,$$
and there exist $H_i\subseteq G_i$ ($1 \le i \le n$) such that 
$$\sum_{i=1}^n a_i|\sum_{j \in H_i} x_j| \ge \frac{1}{2}  \sum_{i=1}^n a_i (\sum_{j \in G_i} |x_j|) = \frac{1}{2} \|x\|.$$
Note that $H_i \in\mathcal{F}$ since $\mathcal{F}$ is hereditary.
 Hence the collection of linear functionals  with a representation  of the form 
$$k=\sum_{i=1}^\infty  a_i 1_{F_i} \in \ell_\infty\qquad (\sum_{i=1}^\infty | a_i|^2 =1,  (F_i)_{i=1}^\infty\, \text{disjoint sets in\,}   \mathcal{F})$$ is a $2$-norming
 set  for $B(\mathcal{F})$.
It follows that the discretized  collection $$K=\{\sum_{r=1}^\infty  \pm 2^{-s(r)} 1_{F_r} \colon \sum_{r=1}^\infty 2^{-2s(r)} \le 1 \}.$$
is  a $4$-norming set.

Let us show that $K \subset \ell_\infty$ is compact in the topology of pointwise convergence on $\ell_\infty$. For $n\ge1$, let
$$ k_n = \sum_{r=1}^\infty 2^{-r}( 1_{U^n_r} -  1_{V^n_r}),$$
where $U^n_r= \cup_{i=1}^{p(n,r)} F^n_i$ and $V^n_r=\cup_{i=1}^{q(n,r)} G^n_i$, and for each $n \ge 1$,  $$\{F^n_{r,i}, G^{n}_{r,j} \colon r\ge1, 1 \le i \le p(n,r), 1 \le j \le q(n,r)\}$$
is a  collection of nonempty disjoint elements of $\mathcal{F}$, and
$$\sum_{r=1}^\infty 2^{-2r}( p(n,r) + q(n,r)) \le 1.$$
In particular, $p(n,r) + q(n,r) \le 2^{2r}$ for all $n,r\ge 1$.
By a diagonal argument,   passing to a subsequence and relabelling, we may assume that
$$ p(n,r) = p_r, q(n,r) = q_r \quad\text{for all $n \ge r$}.$$ By compactness of $\mathcal{F}$, we may also assume that 
$$ \lim_{n \rightarrow \infty} F^n_{r,i} = F_{r,i}\quad (1 \le i \le  p_r),  \lim_{n \rightarrow \infty} G^n_{r,j} = G_{r,j} \quad (1 \le j \le q_r),$$
where  $$\{F_{r,i}, G_{r,j} \colon r\ge1, 1 \le i \le p_r, 1 \le j \le q_r\}$$
is a  collection of  disjoint (possibly empty) elements of $\mathcal{F}$ and
$$\sum_{r=1}^\infty 2^{-2r}( p_r + q_r) \le 1.$$ Set $F_r = \cup_{i=1}^{p_r} F_{r,i}$ and $G_r = \cup_{i=1}^{q_r} G_{r,i}$. It follows that
$$ k = \sum_{r=1}^\infty 2^{-r}(1_{F_r} - 1_{G_r}) \in K$$
and $k = \lim_{n \rightarrow \infty}  k_n$ pointwise in $\ell_\infty$. So $K$ is compact (and metrizable) in the topology of pointwise convergence.

For $x \in B(\mathcal{F})$,  define $\hat{x}\colon K \rightarrow  \mathbb{R}$ by $\hat{x}(k) = k(x).$ Suppose that $(k_n) \subset K$ and $k_n \rightarrow k$ pointwise in $\ell_\infty$. Clearly, 
 $\hat{x}(k_n) \rightarrow \hat{x}(k)$ when $x$ has finite support. Since the finitely supported vectors are norm-dense in $B(\mathcal{F})$, it follows that 
 $\hat{x}(k_n) \rightarrow \hat{x}(k)$ for all $x \in B(\mathcal{F})$, i.e., that   $\hat{x}$ is continuous on $K$.  Since $K$ is $4$-norming for $B(\mathcal{F})$, 
 the mapping  $x \mapsto \hat{x}$ defines a linear isomorphism from $B(\mathcal{F})$ onto a closed subspace of $C(K)$.

Suppose that $(x_n) \subset B(\mathcal{F})$ is  bounded and coordinatewise  null with respect to $(e_n)$. It follows from \eqref{eq: functionalrep} that
$$ \lim_{n\rightarrow \infty} \hat{x}_n(k) = 0 \qquad (k \in K).$$
Hence, by the Riesz representation  and  bounded convergence theorems, $\hat{x}_n \rightarrow 0$ weakly in $C(K)$.
In particular, if $x_n =\sum_{i = p_{n-1}+1}^{p_n}a_i e_i$, where $p_{n-1} < p_n$, is a bounded block basis of $(e_n)$, then $(\hat{x}_n)$ is weakly null in $C(K)$.
Hence $(x_n)$ is weakly null in $B(\mathcal{F})$, which implies that $(e_n)$ is a \textit{shrinking} basis.  On the other hand, since $(e_n)$ satisfies a lower $2$-estimate, it is \textit{boundedly complete}.
It follows from a theorem of James \cite{J2} that $B(\mathcal{F})$ is reflexive.

Next suppose that $\Gamma$ is uncountable. Let $\Gamma_0$ be a countably infinite subset of $\Gamma$. Then
$$X_0 = \{ x \in B(\mathcal{F}) \colon  \operatorname{supp} x \subseteq \Gamma_0\}$$
is the Baernstein space on $\Gamma_0$ corresponding to the family $\mathcal{F}_0 = \{F \cap \Gamma_0 \colon F \in \mathcal{F} \}$. By the first part of the proof, $X_0$ is reflexive. 
But every separable subspace of $B(\mathcal{F})$ is contained in $X_0$  for some $\Gamma_0$. Hence every separable subspace of $B(\mathcal{F})$ is reflexive, which implies that $B(\mathcal{F})$ is also reflexive
since reflexivity is separably determined.

\end{proof}


\begin{thebibliography}{99}
\bibitem{AA} D.E. Alspach and S.A. Argyros, \textit{Complexity of weakly null sequences}, Diss. Math. \textbf{321} (1992), 1--44.
\bibitem {AM} Spiros A. Argyros and Pavlos Motakis, \textit{$\alpha$-Large families and applications to Banach space theory}, Topology and its Applications \textbf{172} (2014),47--67.
\bibitem{B} A. Baernstein, \textit{On reflexivity and summability}, Studia Math. \textbf{42} (1972), 91--94.
\bibitem{BS} Y. Benyamini and T. Starbird, \textit{Embedding weakly compact sets into Hilbert space}, Israel J. Math. \textbf{23} (1970), 137--141.
\bibitem{D} 
M. Day, \textit{Strict convexity and smoothness}, Trans. Amer. Math. Soc. \textbf{78} (1955), 516--528.
\bibitem{DGZ} R. Deville, G. Godefroy, V. Zizler, \textit{Smoothness and renormings in Banach spaces}, 
Monographs and Surveys in Pure and Applied Mathematics, Vol. 64, Pitman, London, 1993.
\bibitem{DKM} S.J. Dilworth, Denka Kutzarova and Pavlos Motakis, \textit{Symmetric $2R$ renormings} (in preparation).
\bibitem{G} Gilles Godefroy, \textit{Renormings of Banach spaces},  Handbook of the geometry of Banach spaces, Vol. I, 781--835, NorthHolland Publishing Co., Amsterdam, 2001.
\bibitem{HJ}  Petr H\'ajek and Michal Johanis, \textit{Characterization of reflexivity by equivalent  renorming}, J. Funct. Anal. \textbf{211} (2004), 163--172.
\bibitem{J} R.C. James, \textit{Reflexivity and the sup of linear functionals}, Israel J. Math. \textbf{13} (1972),
289--300.
\bibitem{J2} R.C. James,\textit{ Bases and reflexivity of Banach spaces}, Ann. Math. \textbf{52} (1950), 518--527.
  \bibitem{KT}  D.N. Kutzarova and S.L.Troyanski, \textit{Reflexive Banach spaces without equivalent norms that are uniformly convex or uniformly differentiable in every direction},  Studia Math. \textbf{72} (1982), 91--95.
\bibitem{M} V.D. Milman, \textit{Geometric theory of Banach spaces. II: geometry of the unit sphere}, Russian
Math. Survey \textbf{26} (1972), 79--163, (transl. from Russian).
\bibitem{OS}
E. Odell and Th. Schlumprecht, \textit{Asymptotic properties of Banach spaces under renormings},
J. Amer. Math. Soc. \textbf{11} (1998), 175--188.



\end{thebibliography}
\end{document}